\documentclass[12pt]{amsart}
\usepackage{amscd, amsfonts, amssymb}
\usepackage{fullpage}
\usepackage{latexsym}
\usepackage{verbatim}
\usepackage{enumitem}
\usepackage[colorlinks=true,citecolor=red,linkcolor=blue]{hyperref}
\setlist[enumerate]{label={\rm(\roman*)},topsep=0pt}

\newcommand\FK{\mathcal{F}_K}
\newcommand\SK{\mathcal{S}_K}
\newcommand\MK{\mathcal{M}_K}
\newcommand\MFK{\mathcal{MF}_K}
\newcommand\PK{\mathcal{P}_K}

\newcommand\sse{\subseteq}

\DeclareMathOperator{\Mat}{Mat}
\DeclareMathOperator{\diag}{diag}
\DeclareMathOperator{\GL}{GL}
\DeclareMathOperator{\SL}{SL}
\DeclareMathOperator{\Stab}{Stab}
\DeclareMathOperator{\SO}{SO}
\DeclareMathOperator{\SP}{Sp}

\numberwithin{equation}{section}

\newtheorem{thm}[equation]{Theorem}
\newtheorem{lem}[equation]{Lemma}
\newtheorem{cor}[equation]{Corollary}
\newtheorem{prop}[equation]{Proposition}
\theoremstyle{definition}
\newtheorem{defn}[equation]{Definition}
\newtheorem{exmp}[equation]{Example}    
\theoremstyle{remark}
\newtheorem{rem}[equation]{Remark} 
\newtheorem{rems}[equation]{Remarks} 


\title{On relative complete reducibility}
\author{Christopher Attenborough, Michael Bate, Maike Gruchot, Alastair Litterick, Gerhard R\"{o}hrle}

\address
{Christopher Attenborough: Department of Mathematics,
University of York,
York YO10 5DD,
United Kingdom}
\email{cea522@york.ac.uk}

\address
{Michael Bate: Department of Mathematics,
University of York,
York YO10 5DD,
United Kingdom}
\email{michael.bate@york.ac.uk}

\address
{Maike Gruchot: Fakult\"at f\"ur Mathematik,
	Ruhr-Universit\"at Bochum,
	D-44780 Bochum, Germany}
\email{maike.gruchot@rub.de}
\curraddr 
{Lehrstuhl D f\"{u}r Mathematik, RWTH Aachen University, 
	D-52062 Aachen, Germany}
\email{maike.gruchot@rwth-aachen.de}

\address
{Alastair Litterick: Fakult\"at f\"ur Mathematik, 
Ruhr-Universit\"at Bochum, 
D-44780 Bochum, Germany, \and
Fakult\"at f\"ur Mathematik, 
Universit\"at Bielefeld, 
D-33501 Bielefeld, Germany
}
\email{alastair.litterick@rub.de}
\curraddr
{Department of Mathematical Sciences, University of Essex, Wivenhoe Park, Colchester, Essex CO4 3SQ, United Kingdom}
\email{ajlitterick@essex.ac.uk}

\address
{Gerhard R\"ohrle: Fakult\"at f\"ur Mathematik,
Ruhr-Universit\"at Bochum,
D-44780 Bochum, Germany}
\email{gerhard.roehrle@rub.de}

\thanks{The first author is supported by an EPSRC Doctoral Training award. The fourth author is supported by the Alexander von Humboldt Foundation.}

\begin{document}

\begin{abstract}
Let $K$ be a reductive subgroup of a reductive group $G$ over an algebraically closed field $k$. The notion of relative complete reducibility, introduced in \cite{BMRT:relative}, gives a purely algebraic description of the closed $K$-orbits in $G^n$, where $K$ acts by simultaneous conjugation on $n$-tuples of elements from $G$. This extends work of Richardson and is also a natural generalization of Serre's notion of $G$-complete reducibility. In this paper we revisit this idea, giving a characterization of relative $G$-complete reducibility which directly generalizes equivalent formulations of $G$-complete reducibility. If the ambient group $G$ is a general linear group, this characterization yields representation-theoretic criteria. Along the way, we extend and generalize several results from \cite{BMRT:relative}.
\end{abstract}

\maketitle
\section{Introduction}
\label{sec:intro}

Let $G$ be a (possibly non-connected) reductive linear algebraic group and let $n\in \mathbb{N}$. The group $G$ acts by simultaneous conjugation on the $n$-fold Cartesian product $G^n$.
In his seminal work \cite[Thm.~16.4]{rich}, Richardson characterized the closed $G$-orbits in $G^n$
in terms of the subgroup structure of $G$.  
In \cite[Thm.~3.1]{BMR} Richardson's characterization 
was shown to be equivalent to a notion of Serre arising from representation theory, \cite{serre2}, 
and these ideas were further extended in \cite{BMRT:relative} to give a characterization of the closed $K$-orbits in $G^n$ for
an arbitrary reductive subgroup $K$ of $G$.
This gave rise to the notion of relative complete reducibility, which we briefly recall now (see Section
\ref{sec:Preliminaries} for full definitions).

Let $H$ be a subgroup of $G$ and let $K$ be a reductive subgroup of $G$.
Recall that (when $G$ is connected) the parabolic subgroups of $G$ have the form $P_\lambda$
where $\lambda$ is a cocharacter of $G$. 
Following \cite{BMRT:relative}, we say that $H$ is
\emph{relatively $G$-completely reducible with respect to $K$} if
for every cocharacter $\lambda$ of $K$
such that $H$ is contained in the subgroup $P_\lambda$ of $G$,  
there exists a cocharacter $\mu$ of $K$  such that
$P_\lambda = P_\mu$ and $H \subseteq L_\mu$,
a Levi subgroup of $P_\lambda$.
For $K = G$, this definition coincides with the usual notion of
$G$-complete reducibility due to Serre, cf.~\cite{BMR},  \cite{serre2}.

The following algebraic characterization of the closed
$K$-orbits in $G^n$ in terms of relative $G$-complete reducibility
was given in \cite[Thm.~1.1]{BMRT:relative}:

\begin{thm}
\label{thm:richardson}
Let $K$ be a reductive subgroup of $G$.
Let $H$ be the algebraic subgroup of $G$ generated by elements
$x_1,\ldots,x_n \in G$.
Then
$K\cdot(x_1, \ldots,x_n)$ is closed in $G^n$
if and only if
$H$ is relatively $G$-completely reducible with respect to $K$.
\end{thm}

Recall that in the particular case $K = G$ and $G = \GL(V)$, a subgroup $H$ of $G$ is $G$-completely reducible if and only if $V$ is a completely reducible $H$-module. This latter property can be defined equivalently by either (i) every $H$-stable flag of subspaces of $V$ admits an $H$-stable complement; (ii) every $H$-stable subspace of $V$ admits an $H$-stable complement; or (iii) $V$ is a direct sum of irreducible $H$-submodules. Our main result, Theorem~\ref{thm:minflags-general} below, generalizes these equivalences firstly to other connected reductive algebraic groups $G$, and also to the relative setting. When $G = \GL(V)$ this characterizes relative $\GL(V)$-complete reducibility (with respect to $K$) in terms of certain flags of submodules of $V$.

In order to state our results, we need some notation. Recall that an R-parabolic subgroup of $G$ is a subgroup of the form $P_{\lambda}$ for a cocharacter $\lambda$ of $G$ (see Section~\ref{sec:Preliminaries}). Let $\mathcal{P}$ be the poset of R-parabolic subgroups of $G$ under inclusion. For a reductive subgroup $K$ of $G$, write $\PK$ for the set of R-parabolic subgroups $P_{\lambda}$ with $\lambda \in Y(K)$. Then $\PK$ is also a poset under inclusion; write $\MK$ for its maximal elements. If $K = G$ is connected then $\PK$ is the poset dual to the spherical building of $G$, and $\MK$ is the set of maximal (proper) parabolic subgroups of $G$; in this case all members of $\MK$ have the same rank (i.e.\ height in the poset $\mathcal{P}$). This need not hold in general, cf.~Example \ref{ex:ConexStabX}.

Conjugation gives a natural action of $G$ and its subgroups on $\mathcal{P}$. Recall that two R-parabolic subgroups are called \emph{opposite} if their intersection is an R-Levi subgroup of $G$. It follows easily from the definitions that a subgroup $H$ of $G$ is relatively $G$-completely reducible with respect to $K$ precisely when each member of $\PK$ containing $H$ has an opposite in $\PK$ containing $H$. The following is our main result.
\begin{thm} \label{thm:minflags-general}
	Let $K \le G$ be reductive algebraic groups with $G$ connected, and let $H$ be a subgroup of $G$. Then the following are equivalent:
	\begin{enumerate}
		\item $H$ is relatively $G$-completely reducible with respect to $K$. \label{minflags-general-i}
		\item Each member of $\MK$ containing $H$ has an opposite in $\MK$ containing $H$. \label{minflags-general-ii}
		\item There is an R-Levi subgroup $L_{\mu}$ with $\mu \in Y(K)$, such that $H \le L_{\mu}$ and $H$ is relatively $L_{\mu}$-irreducible with respect to $K \cap L_{\mu}$. \label{minflags-general-iii}
	\end{enumerate}
\end{thm}

\begin{rems} \leavevmode
	\begin{enumerate}
	\item Conditions \ref{minflags-general-i}--\ref{minflags-general-iii} specialize to the representation-theoretic notions discussed above if we take $K = G = \GL(V)$.
	\item A result of Serre \cite[Thm.~2.2]{serre2} states that when $K = G$ is a connected reductive group, a subgroup is $G$-completely reducible precisely when it lies in a Levi subgroup of each \emph{maximal} parabolic subgroup containing it. This result is central to the study of complete reducibility via the spherical building of $G$, as outlined in \emph{loc.\ cit}. The equivalence of \ref{minflags-general-i} and \ref{minflags-general-ii} in Theorem~\ref{thm:minflags-general} can be recast as a relative analogue of this building-theoretical result, generalizing Serre's result.
	\item In the case $G = \GL(V)$, Theorem~\ref{thm:minflags-general} provides very concrete criteria for relative complete reducibility. The general form of these criteria is given in Corollary~\ref{cor:minflags-glv} below, and is illustrated with explicit examples in Corollary~\ref{cor:classicalK} (when $K = \SO(V)$ or $\SP(V)$) and in Appendix~\ref{sec:appendix} (when $K$ is simple of type $G_{2}$). 
	\end{enumerate}
\end{rems}

The proof of Theorem~\ref{thm:minflags-general} makes use of the fact that $G$ is connected. However, the implications \ref{minflags-general-i} $\Leftrightarrow$ \ref{minflags-general-iii} $\Rightarrow$ \ref{minflags-general-ii} all hold without this assumption. For the missing implication, essentially the only issue is that in a non-connected group, not every R-parabolic subgroup in $\PK$ need arise as the intersection of members of $\MK$ (see Example~\ref{ex:intersection}). However, this does hold when $G$ is connected (Lemma~\ref{lem:intersection}). Theorem~\ref{thm:minflags-general} is therefore a consequence of our next result. For this, define the set
	\[ \PK' := \left\{ P \in \PK \mid P = \bigcap \{Q \in \mathcal{M}_{K} \mid P \le Q\} \right\}. \]

\begin{thm} \label{thm:condition-pk-prime}
	Let $K \le G$ be reductive algebraic groups, and let $H$ be a subgroup of $G$.
	Then each member of $\MK$ containing $H$ has an opposite in $\MK$ containing $H$ if and only if each member of $\PK'$ containing $H$ has an opposite in $\PK'$ containing $H$.
\end{thm}
As a corollary, the conclusion of Theorem~\ref{thm:minflags-general} also holds for (possibly non-connected) reductive groups $G$ such that $\PK = \PK'$.

In the particular case $G = \GL(V)$, Theorem~\ref{thm:minflags-general} gives a representation-theoretic characterization of relative complete reducibility, as follows. Recall that a parabolic subgroup of $\GL(V)$ is the stabilizer of a flag of subspaces in $V$. The poset of flags in $V$ is the dual of the poset of parabolic subgroups in $\GL(V)$, i.e., we set $f\preceq f^\prime$ provided $\Stab_G(f)\supseteq\Stab_G(f^\prime)$. For $K$ a reductive subgroup of $\GL(V)$, we denote by $\mathcal{F}_K$ the set of flags in $V$ which stem from $K$, i.e., which correspond to parabolic subgroups $P_\lambda$ for $\lambda$ a cocharacter of $K$. A flag $f$ in $\FK$ is called \textit{minimal in $\FK$} provided $f^\prime\preceq f$ for $f^\prime$ in $\FK$ implies $f^\prime = f$. Let $\MFK\subseteq \mathcal{F}_K$ be the set of minimal flags in $\mathcal{F}_K$. Again, Example \ref{ex:ConexStabX} shows that members of $\MFK$ may have varying lengths. Of course, $\mathcal{MF}_{\GL(V)}$ is the set of flags of length $1$ in $V$, corresponding to the set of maximal parabolic subgroups in $\GL(V)$.
\begin{cor}
	\label{cor:minflags-glv}
	Let $H$ and $K$ be closed subgroups of $\GL(V)$ with $K$ reductive.
	Then the following are equivalent:
	\begin{enumerate}
		\item $H$ is relatively $\GL(V)$-completely reducible with respect to $K$.
		\item Every $H$-stable flag in $\MFK$ has an $H$-stable opposite in $\MFK$.
		\item There is a maximal torus $S$ of $C_{K}(H)$ such that $H$ preserves the direct-sum decomposition of $V$ into simultaneous $S$-eigenspaces, and this decomposition gives a flag which is maximal among $H$-stable flags in $\FK$.
		 \label{minflags-glv-iii}
	\end{enumerate}
\end{cor}
When $K = G = \GL(V)$, the possibilities for the torus $S$ in \ref{minflags-glv-iii} are products of the centres of the $\GL(V_i)$ corresponding to a direct-sum decomposition of $V$ into irreducible modules $V = V_{1} \oplus \ldots \oplus V_{r}$ $(r \ge 1)$; this gives the usual representation-theoretic characterization.

The next result more closely mirrors the representation-theoretic statement that $V$ is a completely reducible $H$-module if and only if every $H$-submodule has a complement. For this, however, we require an additional hypothesis. Assume once more that $G = \GL(V)$, and write $\mathcal{S}_{K}$ for the set of subspaces of $V$ which arise in flags from $\FK$.
\begin{cor}
\label{cor:maxflags}
Let $H, K$ be subgroups of $\GL(V)$ with $K$ reductive.
Suppose $\mathcal{MF}_K \sse \mathcal{MF}_{\GL(V)}$.
Then the following are equivalent:
\begin{enumerate}
	\item $H$ is relatively $\GL(V)$-completely reducible with respect to $K$. \label{minflagsi}
	\item For each $U\in \SK$ which is stabilized by $H$ there exists $W\in \SK$ such that $H$ stabilizes $W$ 
	and $V=U\oplus W$, as an $H$-module. \label{minflagsii}
\end{enumerate}
\end{cor}
Corollary~\ref{cor:maxflags} readily follows from Corollary~\ref{cor:minflags-glv} and Corollary~\ref{cor:mkvssk}. Note that the implication \ref{minflagsi} $\Rightarrow$ \ref{minflagsii} can fail without the hypothesis on $\MFK$, see Example \ref{ex:ConexStabX}. The essential problem is that an $H$-stable member of $\mathcal{S}_{K}$ need not arise from an $H$-stable flag in $\FK$.

Corollary~\ref{cor:maxflags} may be viewed as a generalization of \cite[Prop.~5.1]{BMRT:relative}. The latter gives a representation-theoretic characterization of relative $\GL(V)$-complete reducibility in the case that $K = \GL(U)$ for a subspace $U$ of $V$, which is closely related to the condition in Corollary~\ref{cor:maxflags}, see Lemma~\ref{lem:GL(U)} and Corollary~\ref{cor:mkvssk}.

Particularly natural candidates for the subgroup $K$ in $\GL(V)$
are the classical groups $\SO(V)$ and $\SP(V)$. This is the theme of our next result,
which gives a characterization of 
relative $\GL(V)$-complete reducibility 
with respect to $\SO(V)$ or $\SP(V)$
in terms of totally singular or totally isotropic subspaces.

\begin{cor}
\label{cor:classicalK}
Let $H$ be a subgroup of $\GL(V)$ and let $K = \SP(V)$ (resp.\ $\SO(V)$). Then the following are equivalent:
\begin{enumerate}
\item
$H$ is relatively $\GL(V)$-completely reducible with respect to $K$. \label{classical-k-i}
\item
Whenever $H$ stabilizes a totally isotropic (resp.\ totally singular) subspace $U$ and its annihilator
$U^{\perp}$, there exists a totally isotropic (resp.\ totally singular) subspace $W \sse V$ such 
that $H$ stabilizes $W$ and $W^{\perp}$, and $V=W\oplus U^{\perp} = U \oplus W^{\perp}$ as $H$-modules. \label{classical-k-ii}
\end{enumerate}
\end{cor}

In the setting of Corollary~\ref{cor:classicalK}, 
flags in $\FK$ have the form 
\begin{equation*}
\label{eq:isotropic}
U_1\sse \ldots  \sse U_r \sse U_r^{\perp}\sse \ldots\sse U_1^{\perp}\sse V,
\end{equation*}
so that minimal flags in $\FK$ are of the form $U\sse U^{\perp} \sse V$ for $U$ a totally isotropic 
(resp.\ totally singular) subspace,
hence the result is immediate from Corollary~\ref{cor:minflags-glv}.

Note that we cannot relax the requirement in Corollary~\ref{cor:classicalK}\ref{classical-k-ii} that $H$ stabilizes $U^\perp$ and $W^{\perp}$, since $H$ does not need to leave the form on $V$ invariant (i.e.\ $H$ need not be a subgroup of $K$).

In Section \ref{sec:rationality}, we present a brief investigation of the notion of relative $G$-complete reducibility
over an arbitrary field, obtaining rational versions of Theorems~\ref{thm:minflags-general} and \ref{thm:condition-pk-prime}. In Appendix~\ref{sec:appendix} we give an illustrative example of Theorem~\ref{thm:minflags-general}, with $K$ simple of exceptional type and $G = \GL(V)$ for a $K$-module $V$.

\section{Preliminaries}
\label{sec:Preliminaries}

We work over a field $k$, which is taken to be algebraically closed except in Section~\ref{sec:rationality}.
Let $G$ be a reductive algebraic group defined over $k$ -- we allow the possibility that $G$ is not connected. 
Let $H$ be a closed subgroup of $G$.
We write $H^\circ$ for the identity component of $H$.

For the set of cocharacters (one-parameter subgroups) of $G$ we write $Y(G)$.
 
Suppose $G$ acts on a variety $X$ and let $x \in X$. 
Then for each cocharacter $\lambda\in Y(G)$ we define a morphism of varieties $\phi_{x,\lambda}:k^{*}\rightarrow X$ via $\phi_{x,\lambda}(a)=\lambda(a)\cdot x$.
If this morphism extends to a morphism $\overline{\phi}_{x,\lambda}: k\to X$, then we say that the limit 
$\lim_{a\to 0} \lambda(a)\cdot x$ exists and set this limit equal $\overline{\phi}_{x,\lambda}(0)$.
For each cocharacter $\lambda\in Y(G)$, let $P_\lambda=\{g\in G\mid \lim_{a\to 0} \lambda(a)g\lambda(a)^{-1} \text{ exists}\}$ and $L_\lambda=\{g\in G\mid \lim_{a\to 0} \lambda(a)g\lambda(a)^{-1}=g\}$.
Following \cite[\S 6]{BMR}, 
we call $P_\lambda$ an \emph{R-parabolic subgroup} of $G$ and $L_\lambda$ an \emph{R-Levi subgroup} of $G$. As mentioned previously, if $G$ is connected then these R-parabolic subgroups and their R-Levi subgroups are precisely the 
parabolic subgroups and their Levi subgroups. 
If $K$ is a reductive subgroup of $G$ and $\lambda\in Y(K)$, we always denote by
$P_\lambda$ the R-parabolic subgroup of $G$ attached to $\lambda$; if we
need to consider the corresponding R-parabolic subgroup of $K$ we write $P_\lambda(K)$ (and similarly for R-Levi subgroups). We recall the following central notions from \cite{BMRT:relative}.
\begin{defn}
\label{def:realtCR}
Let $H$ and $K$ be subgroups of $G$ with $K$ reductive.
We say that \emph{$H$ is relatively $G$-completely reducible with respect to $K$} if, for every $\lambda\in Y(K)$ such that $H$ is contained in $P_\lambda$, there exists $\mu\in Y(K)$ such that $P_\lambda=P_\mu$ and $H$ is contained in $L_\mu$. We sometimes use the abbreviations \emph{relatively $G$-cr with respect to $K$}. We say that \emph{$H$ is relatively $G$-irreducible with respect to $K$} if $H$ is not contained in any subgroup $P_{\lambda}$ with $\lambda \in Y(K)$.
\end{defn}

Note that $H$ is relatively $G$-completely reducible with respect to $K$ if and only if $H$ is relatively $G$-completely reducible with respect to $K^{\circ}$, and similarly for relative $G$-irreducibility.
In the case when $K = G$, Definition~\ref{def:realtCR} coincides with the usual definitions, cf.\ \cite{BMR}.

\begin{lem}
	\label{lem:intersection} Let $K \le G$ be reductive groups with $G$ connected, and let $\PK$, $\MK$ and $\PK'$ be as in the introduction.
\begin{enumerate}
	\item Let $P \in \PK$ and let $Q$ be a maximal R-parabolic subgroup of $G$ containing $P$. Then there exists $P' \in \MK$ such that $P \le P' \le Q$. \label{lem:intersection-i}
	\item Let $P_{\lambda} \in \PK$ with $\lambda(k^{\ast}) \le S$, a fixed maximal torus of $K$. Then $P = \bigcap_{i = 1}^{m} P_{\lambda_i}$ for some $\lambda_i \in Y(S)$ such that $P_{\lambda_i} \in \MK$ for all $i$. In particular, $\PK = \PK'$. \label{lem:intersection-ii}
\end{enumerate}
\end{lem}

\begin{proof} \vspace{-\topsep} \ref{lem:intersection-i} Note that maximal chains of R-parabolic subgroups in the posets $\mathcal{P}$ and $\PK$ are finite; we argue by induction on the height of $P$ in $\PK$. If $P \in \MK$ then there is nothing to prove. Now suppose that $P$ is not maximal in $\PK$. Then there exist $\lambda$, $\nu \in Y(K)$ such that $P = P_{\lambda} \le P_{\nu}$ and $P_{\nu} \in \MK$. If $P_{\nu} \le Q$ then we are done. So we assume that $P_{\nu} \nleq Q$.

Let $\mu\in Y(G)$ such that $Q=P_\mu$. There is a maximal torus $T$ of $G$ in $P_{\lambda}$ such that $T \cap K$ is a maximal torus of $K$ and such that $T \le P_{\lambda} \le P_{\nu}$. Let $T \le B \le P_\lambda$ be a Borel subgroup of $G$.

Let $\Phi$ be the set of roots of $G$ with respect to $T$ and let $\Phi^{+}$ be the set of positive roots of $G$ with respect to $B$, with corresponding set of simple roots $\Delta$. For $I \sse \Delta$, we denote by $P_I$ the standard parabolic subgroup of $G$ generated by $T$ and the root groups corresponding to roots in $-I \cup \Phi^{+}$. By the proof of \cite[Prop.~8.4.5]{spr2}, setting $I_\xi :=\{\alpha\in\Delta\mid \langle \alpha, \xi \rangle =0 \}$ for $\xi \in Y(T)$, where $\left<\alpha,\xi\right>$ denotes the standard pairing between characters and cocharacters of $T$, if $P_{\xi}$ contains $B$ then $P_{I_\xi}=P_\xi$. 
Because of this, we have $I_\lambda \sse I_\nu$ and $I_\lambda \subsetneq I_\mu$. 
Note also that since $P_\mu$ is a maximal parabolic, we have $\Delta\setminus I_\mu = \{\alpha_0\}$, a single simple root.
Since $P_\nu$ is not contained in $P_\mu$, we have $\alpha_0 \in I_\nu$.
Hence $\Delta = I_\mu \cup I_\nu$.

Given any $\alpha \in I_\mu\setminus I_\lambda$, we have $\langle \alpha,\lambda\rangle >0$.
Further, since $I_\nu$ is not contained in $I_\mu$, there is at least one choice of $\alpha\in I_\mu\setminus I_\lambda$ with $\langle \alpha,\nu \rangle >0$ as well. 
Let $n_1,n_2 \in \mathbb{N}$ be such that $\frac{n_1}{n_2} \in \mathbb{Q}_{>0}$ is the maximum of all ratios $\frac{\langle \alpha, \nu\rangle}{\langle \alpha, \lambda\rangle}$ over all $\alpha\in I_\mu \setminus I_\lambda$, and set $\chi := n_1\lambda - n_2\nu$. Note that $\chi \in Y(K)$ and thus $P_{\chi} \in \PK$. We show in the following paragraph that $P_\chi$ is a standard parabolic subgroup of $G$ properly containing $P_\lambda$ and contained in $P_\mu$. This is enough to complete the proof, by induction.

To see that $P_\chi$ is standard, let $\beta\in \Delta$. Then $\beta \in I_\mu$ or $\beta \in I_\nu\setminus I_\mu$, because $\Delta = I_\mu \cup I_\nu$, as observed above.
In the first case, we have $\langle \beta,\chi \rangle \geq 0$, by the choice of $\frac{n_1}{n_2}$ above.
In the second case, $\langle \beta, \nu \rangle = 0$, so we have $\langle \beta,\chi \rangle = n_1\langle \beta, \lambda \rangle \geq 0$.
Hence $\langle \beta,\chi \rangle \geq 0$ for all $\beta \in \Delta$, which shows that $P_\chi$ is standard.
Now, to see that $P_\lambda$ is properly contained in $P_\chi$ it suffices to check that $I_\lambda$ is properly contained in $I_\chi$. 
The containment is clear because $I_\lambda \subseteq I_\nu$ and $\chi$ is a combination of $\lambda$ and $\nu$.
Moreover, if we let $\alpha \in I_\mu\setminus I_\lambda$ be such that the maximum value $\frac{n_1}{n_2}$ above is attained, then $\langle \alpha,\lambda \rangle >0$, but $\langle \alpha,\chi\rangle = 0$, so the containment is proper.
Finally, to see that $P_\chi \subseteq P_\mu$, we only need to see that $\langle \chi,\alpha_0 \rangle >0$, where $\alpha_0\in \Delta$ is the unique simple root outside $I_\mu$, as above.
But we have argued that $\alpha_0\in I_\nu$, and $\alpha_0 \not\in I_\lambda$ since $P_\lambda$ is proper in $P_\mu$.
Hence $\langle \alpha_0,\chi\rangle = n_1\langle \alpha_0,\lambda\rangle >0$, and we are done. 

Thus, by induction there exists $P' \in \MK$ such that $P_\chi \le P' \le P_{\mu} = Q$, which completes the proof of (i). 

\ref{lem:intersection-ii} Since $G$ is connected, the R-parabolic subgroup $P_{\lambda}$ is equal to the intersection of all maximal R-parabolic subgroups of $G$ which contain it. By \ref{lem:intersection-i}, for each such maximal R-parabolic subgroup $Q_{i}$ $(i = 1,\ldots,m)$ we can find an R-parabolic subgroup $P_{i} \in \MK$ with $P_{\lambda} \le P_{i} \le Q_{i}$. Then $P_\lambda = \bigcap_{i = 1}^{m} P_{i}$. Each subgroup $P_{i}$ has the form $P_{\lambda_i}$ for some cocharacter $\lambda_i$ of $K$. Now $S \le P_{\lambda} \le P_{\lambda_i}$ for each $i$, so $S$ is a maximal torus of $P_{\lambda_i}(K)$. Since $\lambda_i(k^{\ast})$ is a sub-torus of $P_{\lambda_i}(K)$, it is $P_{\lambda_i}(K)$-conjugate to a subgroup of $S$. So we can replace $\lambda_i$ by an appropriate $P_{\lambda_i}$-conjugate so that $\lambda_i \in Y(S)$ for each $i$, as required.
\end{proof}

The following example shows that Lemma~\ref{lem:intersection}\ref{lem:intersection-ii} can fail when $G$ is not connected. %Note, however, that this does not furnish us with a counterexample to the non-connected analogue of Theorem~\ref{thm:minflags-general}.
\begin{exmp} \label{ex:intersection}
Let $T$ be a $1$-dimensional torus, let $\left<x\right>$ be cyclic of order $8$ and let $G = K = T^{8} \rtimes \left<x\right>$, with $x$ permuting the factors of $T^{8}$ in the obvious manner. Then each R-parabolic subgroup of $G$ contains $T^{8} = G^{\circ}$, and every subgroup $G^{\circ} \le P \le G$ arises as an R-parabolic subgroup $P_{\lambda}$, depending on whether the $1$-dimensional torus $\lambda(k^{\ast})$ is centralized by $x$, $x^{2}$, $x^{4}$ or none of these. Then $T^{8} \rtimes \left<x^{2}\right>$ is the unique maximal proper R-parabolic subgroup of $G$, and in particular its subgroup $T^{8} \rtimes \left<x^{4}\right>$ is not the intersection of the maximal R-parabolic subgroups of $G$ containing it.
\end{exmp}

\begin{lem} \label{lem:IntersectionParabolic}
	Let $G$ be reductive, and let $\lambda_i$ $(i = 1,\ldots,m)$ be pairwise commuting cocharacters of $G$ such that there exists a Borel subgroup $B$ of $G$ with $B \subseteq P_{\lambda_{i}}$ for each $i$. Then there exist postive integers $n_i$ $(i = 1,\ldots,m)$ such that $P_{\lambda} = \bigcap_{i = 1}^{m} P_{\lambda_i}$, where $\lambda = \sum_{i = 1}^{m} n_i \lambda_i$.
\end{lem}

\begin{proof} 
	If $G$ is connected, this follows quickly from the well-understood theory of standard parabolic subgroups (parabolic subgroups containing a fixed Borel subgroup); the subgroups $P_{\lambda_i}$ correspond to choosing subsets of simple roots of $G$, and $P_{\lambda}$ corresponds to choosing the union of these sets (independently of the choice of positive integers $n_i$). For general $G$, by \cite[Lem.~6.2(iii)]{BMR} it suffices to show that $R_u(P_\lambda) = R_u\left( \bigcap_{i = 1}^{m} P_{\lambda_i}\right)$ and $L_{\lambda} = \left(\bigcap_{i = 1}^{m} L_{\lambda_i}\right)$ for some choice of integers $n_i$. Moreover, it suffices to treat the case $m = 2$, and the general case then follows by an easy induction. Now, since $R_u(P) = R_u(P^{\circ}) = R_u(P \cap G^{\circ})$ for every R-parabolic subgroup $P$, the equality $R_u(P_\lambda) = R_u\left( P_{\lambda_1} \cap P_{\lambda_2} \right)$ follows from the corresponding result for connected $G$ and \cite[Lem.~6.2(iii)]{BMR}. Finally \cite[Lem.~6.2(i)]{BMR} tells us that the equality $L_\lambda = L_{\lambda_1} \cap L_{\lambda_2}$ holds for sufficiently large $n_1$.
\end{proof}

The following example shows that the positive integers $n_i$ cannot be chosen arbitrarily when $G$ is not connected. We thank Dr.\ Tomohiro Uchiyama for pointing this out.
\begin{exmp}
Let $G = \SL_3(k)\!\left<\sigma\right>$, where $\sigma$ is the inverse-transpose automorphism composed with conjugation by $\begin{pmatrix}0 & 0 & 1 \\ 0 & 1 & 0 \\ 1 & 0 & 0\end{pmatrix}$. Then $\sigma$ normalises the diagonal maximal torus $T$ and the upper triangular Borel subgroup $B$, swapping the root groups $U_{\alpha}$ and $U_{\beta}$ corresponding to the $(1,2)$ and $(2,3)$ matrix coordinates. Let $\lambda(c) = \diag(c,c,c^{-2})$, $\mu(c) = \diag(c^{2},c^{-1},c^{-1})$ for $c \in k^{\ast}$, so that $\lambda,\mu \in Y(T)$ with $P_{\lambda}^{\circ} = \left<T, U_{\pm \alpha}, U_{\beta}\right>$ and $P_{\mu}^{\circ} = \left<T,U_{\pm \beta},U_{\alpha}\right>$. It is evident that $\sigma$ does not normalise $P_{\lambda}^{\circ}$ or $P_{\mu}^{\circ}$, hence is not contained in $P_{\lambda}$ or in $P_{\mu}$. On the other hand, it is easily checked that $\sigma$ centralises $(\lambda+\mu)(k^{\ast})$; in particular $P_{\lambda + \mu}$ contains $\sigma$ hence is strictly larger than $P_{\lambda} \cap P_{\mu}$. On the other hand, for any choice of \emph{distinct} positive integers $n_1,n_2$, it is the case that $P_{n_1 \lambda + n_2\mu} = P_{\lambda } \cap P_{\mu} = B$.
\end{exmp}

We require the following useful result, which is \cite[Lem.~3.3]{BMRT:relative}.
\begin{lem}
	\label{lem:3.3}
	Let $K \le G$ be reductive groups.
	\begin{enumerate}
		\item
		Let $\lambda, \mu\in Y(K)$ such that $P_\lambda=P_\mu$ and 
		$u\in R_u(P_\lambda(K))$ such that $uL_\lambda(K)u^{-1}=L_\mu(K)$.
		Then $uL_\lambda u^{-1}=L_\mu$.
		\item Let $H$ be a subgroup of $G$.
		Then $H$ is relatively $G$-completely reducible with respect to $K$ if
		and only if for every $\lambda\in Y(K)$ such that $H\subseteq P_\lambda$
		there exists $u\in R_u(P_\lambda(K))$ such that 
		$H\subseteq L_{u\cdot \lambda}$. \label{lem:3.3-ii}
	\end{enumerate}
\end{lem}

The following is the final ingredient needed in the proof of Theorems~\ref{thm:minflags-general} and \ref{thm:condition-pk-prime}.
\begin{lem} \label{lem:correct-sets}
Let $K \le G$ be reductive groups. If $P \in \MK$ (resp.\ $P \in \PK'$) and $Q \in \PK$ is opposite to $P$, then $Q \in \MK$ (resp.\ $Q \in \PK'$).
\end{lem}

\begin{proof}
We prove firstly that if $P \in \MK$, and if $Q \in \PK$ is opposite to $P$, then $Q \in \MK$. Suppose not. Then there exists $Q' \in \MK$ such that $Q \lneq Q'$. Let $T$ be a maximal torus of $K$ contained in $Q$ and $Q'$. Then there exits $\lambda \in Y(T)$ such that $P = P_{\lambda}$ and $Q = P_{-\lambda}$. Since $T \le P_{-\lambda} < Q'$, there exists $\mu \in Y(T)$ such that $Q' = P_{\mu}$. Let $P' = P_{-\mu} \in \PK$. Then $H \le P_{\lambda} < P_{-\mu}$, contradicting $P \in \MK$.

Now suppose $P \in \PK'$. Since $G$ is noetherian, we can write $P$ as a finite intersection $P = \bigcap_{i = 1}^{m} P_{i}$ with $P_{i} \in \MK$ for each $i$. Let $Q \in \PK$ be opposite to $P$, and let $T$ be a maximal torus of $K$ contained in $P \cap Q$. Then for each $i$ we have $P_{i} = P_{\lambda_i}$ for some $\lambda_{i} \in Y(T)$, and by Lemma~\ref{lem:IntersectionParabolic} we have $P = P_{\lambda}$ where $\lambda = \sum_{i = 1}^{m} n_i \lambda_i$ for some positive integers $n_i$. By Lemma~\ref{lem:3.3} and the uniqueness of opposite R-parabolic subgroups containing a given R-Levi subgroup \cite[Lem.~6.11]{BMR}, we can conjugate $\lambda$ (and each $\lambda_i$) by some fixed element of $R_{u}(P_{\lambda}(K))$ such that $Q = P_{-\lambda} = \bigcap_{i = 1}^{m} P_{-\lambda_i}$. By the paragraph above, each subgroup $P_{-\lambda_i} \in \MK$ for each $i$, hence $Q \in \PK'$.
\end{proof}

\section{Proof of Theorems \ref{thm:minflags-general} and \ref{thm:condition-pk-prime}} \label{sec: proof}

We now prove Theorems~\ref{thm:minflags-general} and \ref{thm:condition-pk-prime}. We begin with the proof of Theorem~\ref{thm:condition-pk-prime}, from which Theorem~\ref{thm:minflags-general} follows in short order.

Recall that $G$ is a (not necessarily connected) reductive algebraic group and $H$ and $K$ are subgroups of $G$, with $K$ also reductive (but not necessarily connected). Also $\PK$ is the set of R-parabolic subgroups of $G$ of the form $P_{\lambda}$ with $\lambda \in Y(K)$, $\MK$ is the set of maximal members of $\PK$ under inclusion, and $\PK'$ is the set of members of $\PK$ which can be realized as intersections of members of $\MK$.

\begin{proof}
[Proof of Theorem~\ref{thm:condition-pk-prime}]
Suppose that each member of $\PK'$ which contains $H$ admits an opposite in $\PK'$ containing $H$. Since $\MK \subseteq \PK'$ by definition, if $P \in \MK$ and $P \ge H$ then $P \in \PK'$, hence admits an opposite $Q \in \PK'$ which contains $H$. By Lemma~\ref{lem:correct-sets} we have $Q \in \MK$, as required.

Conversely, suppose that each member of $\MK$ containing $H$ admits an opposite in $\MK$ which contains $H$, and let $P \in \PK'$ such that $P \ge H$. By definition of $\PK'$ we can write $P = \bigcap_{i = 1}^{m} P_{i}$ for some $m \in \mathbb{N}$ and $P_{i} \in \MK$ for each $i$. Taking $m$ to be minimal, we proceed by induction on $m$, the case $m = 1$ being our starting hypothesis.

By Lemma~\ref{lem:3.3} we can fix a maximal torus $T$ of $K$ contained in $P$ such that $P = P_{\lambda}$ and $P_{i} = P_{\lambda_i}$ for some cocharacters $\lambda$ and $\lambda_i \in Y(T)$ for each $i$. Write this intersection as $P_{\lambda} = P_{\lambda_1} \cap \bigcap_{i = 2}^{m} P_{\lambda_i}$. By Lemma~\ref{lem:IntersectionParabolic} we are free to write $\bigcap_{i = 2}^{m} P_{\lambda_i} = P_{\nu}$ where $\nu = n_2\lambda_2 + \ldots + n_m\lambda_m$ for some positive integers $n_i$, and we are also free to replace $\lambda$ (without changing $P_{\lambda}$) such that $\lambda = n_1 \lambda_1 + n_\nu \nu$ for some positive integers $n_1$, $n_\nu$.

Since $P_{\lambda_1} \in \MK$, by hypothesis $H$ lies in some R-Levi subgroup of $P_{\lambda_1}$. Since $P_{\lambda} \le P_{\lambda_1}$ we have $R_{u}(P_{\lambda_1}(K)) \le R_{u}(P_{\lambda}(K))$. Thus we are free to replace $H$ by an $R_{u}(P_{\lambda_1}(K))$-conjugate so that $H \le L_{\lambda_1}$, and this does not change whether $H$ lies in an R-Levi subgroup of $P_{\lambda}$ corresponding to a cocharacter of $K$. Moreover, replacing $H$ with such a conjugate, we still have $H \le P_{\lambda} \le P_{\nu}$. Now since $H \le L_{\lambda_1}$, we have $\lambda_1 \in Y(C_{K}(H) \cap L_{\nu}(K)) = Y(C_{L_\nu(K)}(H))$.

Now, we can apply the induction hypothesis to $P_{\nu}$, so that there exists $\sigma \in Y(K)$ such that $P_{\sigma} \in \PK'$ is opposite to $P_{\nu}$ and $H \le L_{\sigma}$. Then there exists $u \in R_{u}(P_{\nu}(K))$ such that $u \cdot \nu \in Y(L_{\sigma}(K))$. So we set $\tau = -(u \cdot \nu)$ as cocharacters of some maximal torus of $L_{\sigma}(K)$. By the uniqueness of opposite parabolic subgroups containing a fixed maximal torus \cite[Lem.~6.11]{BMR} we have $P_{\tau} = P_{\sigma}$, $L_{\tau} = L_{\sigma}$. Since $\tau$, $\lambda_1 \in Y(C_{P_\nu(K)}(H))$, $\tau$ is $C_{P_\nu(K)}(H)$-conjugate to a cocharacter $\rho$ which commutes with $\lambda_1$. Then $P_{\rho} \cap P_{\nu} = P_{\tau} \cap P_{\nu}$ so $P_{\rho}$ is still opposite to $P_{\nu}$, and $H \le L_{\rho}$. Now $\rho$ and $\nu$ both commute with $\lambda_1$, and so the images of these cocharacters are all contained in some maximal torus of $K$, call it $S$. Since $\lambda = n_1 \lambda_1 + n_\nu \nu$, the image of $\lambda$ also lies in $S$.

Now, again by the uniqueness of opposite R-parabolic subgroups among those containing a given maximal torus of $K$, it follows that we can scale $\rho$ so that $\rho = -n_\nu\nu$ and thus $-\lambda = -n_1\lambda_1 - n_\nu \nu = -n_1\lambda_1 + \rho$ as elements of $Y(S)$. This shows that $P_{-\lambda} = P_{-\lambda_1} \cap P_{\rho}$ is opposite to $P_{\lambda}$, lies in $\PK'$ and contains $H$, as required.
\end{proof}

\begin{proof}
[Proof of Theorem~\ref{thm:minflags-general}]
The equivalence of the conditions \ref{minflags-general-i} and \ref{minflags-general-iii} follows immediately from \cite[Prop.~3.17(ii)]{BMRT:relative}. Now suppose \ref{minflags-general-i} holds and let $P \in \MK$ with $P \ge H$. By hypothesis, there exists an R-parabolic subgroup $Q \in \PK$ which is opposite to $P$ and contains $H$, and then $Q \in \MK$ by Lemma~\ref{lem:correct-sets}, so condition \ref{minflags-general-ii} holds. Note that we have not yet used the hypothesis that $G$ is connected.

Finally, suppose that \ref{minflags-general-ii} holds, so that each $P \in \MK$ which contains $H$ admits an opposite in $\MK$ which contains $H$. Since $G$ is connected by hypothesis, we have $\PK = \PK'$ by Lemma~\ref{lem:intersection}\ref{lem:intersection-ii}. Then Theorem~\ref{thm:condition-pk-prime} tells us that each $P \in \PK$ containing $H$ has an opposite in $\PK$ containing $H$, hence condition \ref{minflags-general-i} holds.
\end{proof}

\section{The case \texorpdfstring{$G = \GL(V)$}{G = GL(V)}}

We take this opportunity to note some interesting special cases of our results in the case $G = \GL(V)$. The first follows immediately from Lemma~\ref{lem:intersection}\ref{lem:intersection-i}. Recall that $\FK$ denotes the set of flags arising from parabolic subgroups corresponding to cocharacters of $K$, and $\mathcal{MF}_{K}$ denotes the minimal members of $\FK$ (i.e. those whose stabilizers lie in $\MK$).
\begin{cor}
	\label{cor:appearingU}
	Let $G = \GL(V)$, let $K \le G$ be a reductive subgroup, let $f$ be a flag in $\FK$ and let $U$ be a subspace in $f$.
	Then there is a flag $f^\prime$ in $\mathcal{MF}_K$ such that $f^\prime\preceq f$ and $U$ appears in $f^\prime$. 
\end{cor}

Recall that $\SK$ is the set of subspaces of $V$ which appear in flags from $\mathcal{F}_K$. The following is an immediate consequence of Corollary~\ref{cor:appearingU}.
\begin{cor}
\label{cor:mkvssk}
Let $K$ be a reductive subgroup of $\GL(V)$.
Then the following are equivalent:
\begin{enumerate}
	\item $\SK = \{ U\sse V\mid (U\sse V)\in\mathcal{F}_K\}$. \label{mkvssk-i}
	\item $\mathcal{MF}_K \sse \mathcal{MF}_{\GL(V)}$. \label{mkvssk-ii}
\end{enumerate}
\end{cor}

Next, we note that the implication \ref{minflagsi} $\Rightarrow$ \ref{minflagsii} of Corollary~\ref{cor:maxflags} fails without the hypothesis on $\MFK$, as the following example illustrates.
\begin{exmp}
	\label{ex:ConexStabX}
Let $G=\GL_4(k)$ and let $K$ be the subgroup of diagonal matrices of the form 
$\diag(t,s, s^{-1},t^{-1})$ with $s,t \in k^*$.
Let $e_1,\ldots, e_4$ be the standard basis of $k^4$ and $U=\langle e_1, e_2, e_3 \rangle$.
Suppose that $H$ is the parabolic subgroup of $G$ corresponding to the flag $U\sse V$.
Since the flags from $\mathcal{F}_K$ have subspaces of dimension $(2,4)$, $(1,3,4)$ and $(1,2,3,4)$, the group $H$ is not contained in $P_\lambda$ for any $\lambda\in Y(K)\setminus \{1\}$.
Hence trivially, $H$ is relatively $G$-cr with respect to $K$.
Note that $U\in \SK$ and $H$ stabilizes $U$.
One checks that the complement to $U$ in the set $\SK$ is $W=\langle e_4 \rangle$.
But $H$ does not stabilize $W$.
\end{exmp}

\begin{rem}
Let $G = \GL(V)$. If $\mathcal{MF}_{K} \not\sse \mathcal{MF}_{G}$ then there exists a subgroup $H$ of $G$ such that $H$ is relatively $G$-cr with respect to $K$ and $H$ stabilizes a subspace $U'\in \SK$ but does not stabilize any complement to $U'$.
To see this, note that since $\SK\not=\{ U\sse V\mid (U\sse V)\in\mathcal{F}_K\}$, by Corollary~\ref{cor:mkvssk},
there exists a $U'$ in $\SK$ such that $(U'\sse V)\not\in\mathcal{F}_K$.
Set $H:=\Stab_G(U'\sse V)$.
Then $H$ is not contained in $P_\lambda$ for any $\lambda\in Y(K)\setminus \{1\}$.
Trivially, $H$ is relatively $G$-cr with respect to $K$.
Note that $H$ stabilizes $U'$ in $\SK$ but does not stabilize any complement to $U'$, since $H$ is a maximal parabolic subgroup of $G$.
\end{rem}

For a fixed subspace $U\sse V$, we  show 
in the following lemma that $K = \GL(U)$ satisfies 
the condition in Corollary~\ref{cor:mkvssk}(i).
A maximal torus of $G$ also satisfies the condition.
So Corollary~\ref{cor:maxflags} applies in these instances,
thanks to Corollary~\ref{cor:mkvssk}.

\begin{lem}
\label{lem:GL(U)}
Let $G = \GL(V)$ and let $U \sse V$. Fix a complement $\widetilde{U}$ to $U$ in $V$.
Let $K=\GL(U) \le G$, embedded via the decomposition $V = U\oplus\widetilde{U}$.
Then $\SK=\{ W\sse V\mid (W\sse V)\in\mathcal{F}_K\}$.
\end{lem}

\begin{proof} 
Let $(W_1\sse\ldots\sse W_m\sse V)$ be in $\mathcal{F}_K$.
One sees by inspection that for each $1\leq i\leq m$ we have $W_i\sse U$ or $\widetilde{U}\sse W_i$.
On the other hand, suppose that $W$ is a subspace contained in $U$.
Then we can find a complement $W'$ to $W$ containing $\widetilde{U}$ and the cocharacter which acts with weight $1$ on $W$ and weight $0$ on $W'$ lies in $Y(K)$ and affords the flag $(W\subseteq V)\in \mathcal{F}_K$.
Similarly, all flags $(W\subseteq V)$ with $\widetilde{U}\subseteq W$ are in $\mathcal{F}_K$.  
Hence 
\begin{equation*}
\mathcal{F}_{K}=\{ (W_1\sse\ldots\sse W_m\sse V)\in\mathcal{F}_G \mid W_i\sse U \text{ or } \widetilde{U}\sse W_i \text{ for } 1\leq i\leq m\text{, for some }m\},
\end{equation*}
and so 
\begin{equation}
\label{eq:comp}
\SK=\{ W'\sse V\mid W'\sse U \text{ or } \widetilde{U}\sse W'\}=\{ W\sse V\mid (W\sse V)\in\mathcal{F}_K\}, 
\end{equation}
as claimed.
\end{proof}

In view of \eqref{eq:comp} and Corollary~\ref{cor:mkvssk}, Corollary~ \ref{cor:maxflags} and Lemma \ref{lem:GL(U)} 
imply \cite[Prop.~5.1]{BMRT:relative}.
So Corollary~\ref{cor:maxflags} may be viewed as a generalization of 
the special case treated in 
\cite[Prop.~5.1]{BMRT:relative}.

Corollary~\ref{cor:classicalK} considers situations when $K$ acts irreducibly on $V$. We close this section with a characterization of relative $\GL(V)$-complete reducibility in case $V$ decomposes as a direct sum of $K$-modules. We first note a consequence of Lemma~\ref{lem:3.3} which characterizes relative $G$-complete reducibility when $G$ and $K$ admit compatible direct-product structures.

\begin{cor}
	\label{cor:produkt}
	For $i = 1,2$, let $K_i \subseteq  G_i$ be reductive groups,
	$G:=G_1\times G_2$ and $K:=K_1\times K_2$. Let $H\subseteq G$ be a subgroup.
	Then $H$ is relatively $G$-completely reducible with respect to $K$ if and only if $H$ is relatively $G$-completely reducible with respect to $K_i$ for $i=1,2$.
\end{cor}

\begin{proof}
Let $\lambda\in Y(K_1)$ such that $H \le P_\lambda$.
By the proof of \cite[Lem.~2.12]{BMR}, the parabolic subgroups of $G$  
arising from cocharacters of $K$ have the form $P_{\lambda_1}\times P_{\lambda_2}$ with 
$\lambda_i\in Y(K_i)$ for $i=1,2$, since $G=G_1\times G_2$ and $K=K_1\times K_2$.
Hence $P_\lambda=P_\lambda(G_1)\times G_2$.
Since $H$ is relatively $G$-completely reducible with respect to $K$, there
exists a $u=(u_1,u_2)\in R_u(P_\lambda(K))$ such that
$H \le L_{u\cdot \lambda}=L_{u_1\cdot \lambda}(G_1)\times G_2$,
by Lemma \ref{lem:3.3}\ref{lem:3.3-ii}.
Therefore, $u_1\in R_u(P_\lambda(K_1))$. It follows that $H$ is relatively $G$-cr
with respect to $K_1$, by Lemma \ref{lem:3.3}\ref{lem:3.3-ii}.
The proof for $K_2$ is analogous.
	
For the reverse implication let $\lambda = (\lambda_1,\lambda_2)\in Y(K)=Y(K_1)\times Y(K_2)$ such that $H \le P_{\lambda_1}\times P_{\lambda_2}$.
Since $H$ is relatively $G$-completely reducible with respect to $K_i$ for
$i=1,2$, there exits a $u_i\in R_u(P_{\lambda_i}(K_i))$ such that
$H \le L_{u_1\cdot\lambda_1}(G_1)\times G_2$
resp.~$H \le G_1\times L_{u_2\cdot\lambda_2}(G_2)$.
Therefore, we obtain 
$$H \le (L_{u_1\cdot\lambda_1}(G_1)\times G_2)\cap (G_1\times L_{u_2\cdot\lambda_2}(G_2)) =L_{u\cdot\lambda}$$ 
for $u=(u_1,u_2)\in R_u(P_\lambda(K))$.
Once again, by Lemma \ref{lem:3.3}\ref{lem:3.3-ii}, $H$ is relatively $G$-cr with respect to $K$.
\end{proof}

The following result is now immediate from Corollary~\ref{cor:produkt} and \cite[Cor.~3.6]{BMRT:relative}.

\begin{cor} \label{cor:product}
Let $G = \GL(V)$ and suppose that both $H$ and $K$ preserve a direct-sum decomposition $V = \bigoplus_{i = 1}^{n} V_{i}$. Suppose also that $K = K_1 \times \cdots \times K_n$ where $K_i \le \GL(V_i)$ for each $i$. Then $H$ is relatively $G$-completely reducible with respect to $K$ if and only if $H$ is relatively $G$-completely reducible with respect to $K_i$ for all $i$.
\end{cor}

Both implications in Corollary~\ref{cor:product} fail in general, as illustrated by our next example.

\begin{exmp}
Let $G=\GL(k^4)$, $K=\{\diag(t, s, s^{-1}, t^{-1})\mid t,s \in k^*\}$,
and $\{e_1, e_2, e_3, e_4\}$ is the canonical basis for $k^4$.
Set $V_1=\langle e_1, e_2\rangle$ and $V_2=\langle e_3, e_4\rangle$.
Let $K_i$ be the image of the projection from $K$ to $\GL(V_i)$ for $i=1,2$.

Let $H$ be the stabilizer of $U:=\langle e_2, e_4 \rangle$ in $G$ and note that  $(U\subseteq k^4)$ belongs to $\mathcal{F}_K$.
Thus $H$ is a maximal parabolic subgroup of $G$ and corresponds to a cocharacter of $K$, and as such it is not relatively $G$-cr with respect to $K$.
However, $H$ does not correspond to a cocharacter of $K_i$, and by maximality $H$ is not contained in any parabolic subgroup of $G$ correspond to a cocharacter of $K_i$ ($i = 1,2$). Hence $H$ is relatively $G$-irreducible with respect to  
$K_i$ so it is relatively $G$-cr with respect to $K_i$,  for $i=1,2$.

Now let $\widetilde{H}$ be the stabilizer of $\widetilde{U}:=\langle e_1\rangle$ in $G$.
Note that $\widetilde{H}$ is a maximal parabolic subgroup of $G$ and corresponds to a cocharacter of $K_1$, thus it 
is not relatively $G$-cr with respect to $K_1$.
However, since $\widetilde{H}$ does not correspond to a cocharacter of $K$, it is 
relatively $G$-irreducible with respect to $K$, in particular it is 
relatively $G$-cr with respect to $K$.
\end{exmp}

\section{Rationality Questions}
\label{sec:rationality}
In this section $k$ denotes an arbitrary field, $G$ is a reductive $k$-defined group and $K$ is a reductive $k$-defined subgroup of $G$. For a $k$-defined closed subgroup $M$ of $G$, write $Y_{k}(M)$ for the $k$-defined cocharacters of $M$, and let $M(k)$ denote the group of $k$-points of $M$. First, we recall the definition of relative $G$-complete reducibility over $k$ from \cite[Def.~4.1]{BMRT:relative}, and also define the analogue of relative $G$-irreducibility over $k$.

\begin{defn}
Let $H$ be a subgroup of $G$.
We say that \emph{$H$ is relatively $G$-completely reducible over $k$ with respect to $K$} if for every $\lambda\in Y(K)$ such that $P_\lambda$ is $k$-defined and $H$ is contained in $P_\lambda$, there exists $\mu\in Y(K)$ such that $P_\lambda=P_\mu$,  $H$ is contained in $L_\mu$ and $L_\mu$ is $k$-defined. We also say that \emph{$H$ is relatively $G$-irreducible over $k$ with respect to $K$} if $H$ is not contained in any $k$-defined parabolic subgroup $P_{\lambda}$ with $\lambda \in Y(K)$.
\end{defn}

\begin{rem} \label{rem:not-all-k-def-parabs}
By \cite[Lem.~4.8]{BMRT:relative}, a subgroup is relatively $G$-cr over $k$ with respect to $K$ if and only if for every $\lambda \in Y_{k}(K)$ such that $H \le P_\lambda$, there exists $\mu \in Y_{k}(K)$ such that $P_\lambda = P_\mu$ and $H \le L_{\mu}$. By identical arguments to those in the proof of \cite[Lem.~4.8]{BMRT:relative}, a subgroup is relatively $G$-irreducible over $k$ with respect to $K$ if and only if it is not contained in any R-parabolic subgroup $P_{\lambda}$ with $\lambda \in Y_{k}(K)$.
\end{rem}

Analogous to Theorem \ref{thm:richardson}, we have a geometric characterization of relative $G$-complete reducibility over $k$. We recall some definitions \cite[Def.~1.1]{BHMR:cochars}, \cite[Def.~5.4]{BMRT:S-instability}.
\begin{defn} Let $G$ be reductive and $k$-defined.
\begin{enumerate}
	\item Let $G$ act $k$-morphically on an affine $k$-variety $X$, and let $x \in X$. The orbit $G(k)\cdot x$ is called \emph{cocharacter-closed over $k$} if for all $\lambda\in Y_k(G)$ such that $\lim_{a\to 0} \lambda(a)\cdot x$ exists, then this limit lies in the $G(k)$-orbit $G(k) \cdot x$.
	\item A \emph{generic tuple} for a subgroup $H$ of $G$ is an $n$-tuple $(h_1,\ldots,h_n) \in G^{n}$ such that $H$ and $\{h_1,\ldots,h_n\}$ generate the same associative subalgebra of $\Mat_{m \times m}(\overline{k})$, for some embedding $G \to \GL_{m}(\overline{k})$. 
\end{enumerate}
\end{defn}
Note that a generic tuple for $H$ always exists since $\Mat_{m \times m}(\overline{k})$ is a finite-dimensional $\overline{k}$-algebra. The following summarizes part of \cite[Thm.~4.12(iii)]{BMRT:relative}.

\begin{thm}
\label{thm:cochar-realCR}
Let $K$ be a reductive subgroup of $G$, let $H \le G$ and let $\mathbf{h} \in G^{n}$ be a generic tuple for $H$. Then $H$ is relatively $G$-completely reducible over $k$ with respect to $K$ if and only if $K(k) \cdot \mathbf{h}$ is cocharacter-closed over $k$.
\end{thm}

We now generalize Theorems~\ref{thm:minflags-general} and \ref{thm:condition-pk-prime} to the rational setting. First, we need a rational analogue of \cite[Prop.~3.17]{BMRT:relative}.

\begin{prop} \label{prop:rational-317}
Let $K$ be a $k$-defined reductive subgroup of $G$, and let $H \le G$.
\begin{enumerate}
	\item Suppose that $H \le L := L_{\lambda}$ where $\lambda \in Y_{k}(K)$. Then $H$ is relatively $G$-completely reducible over $k$ with respect to $K$ if and only if $H$ is relatively $L$-completely reducible over $k$ with respect to $K \cap L$. \label{rational-317-i}
	\item If $L$ is minimal among subgroups of the form $L_{\lambda}$ with $\lambda \in Y_{k}(K)$ containing $H$, then $H$ is relatively $G$-completely reducible over $k$ with respect to $K$ if and only if $H$ is relatively $L$-irreducible over $k$ with respect to $K \cap L$. \label{rational-317-ii}
\end{enumerate}
\end{prop}

\begin{proof} \vspace{-\topsep}
	Noting that the image of $\lambda$ is a $k$-split torus, part \ref{rational-317-i} follows directly from \cite[Thm.~5.4(ii)]{BHMR:cochars} applied to the $K$-orbit of $\mathbf{h} \in G^{n}$, where $\mathbf{h} \in G^{n}$ is a generic tuple for $H$. For part \ref{rational-317-ii}, we have shown that $H$ is relatively $L$-cr over $k$ with respect to $K \cap L$. But also, from the minimality of $L$, it follows that any $k$-defined cocharacter of $K \cap L$ centralized by $H$ is central in $L$, hence $H$ is relatively $L$-irreducible over $k$ with respect to $K \cap L$.
\end{proof}

We are now in a position to prove the rational counterparts of Theorems~\ref{thm:minflags-general} and \ref{thm:condition-pk-prime}. Define the following analogues of the sets $\PK$, $\MK$ and $\PK'$:
\begin{align*}
\mathcal{P}_{K,k} &= \left\{ P_{\mu} \mid \mu \in Y_{k}(K) \right\},\\
\mathcal{M}_{K,k} &= \left\{ \text{inclusion-maximal members of } \mathcal{P}_{K,k} \right\},\\
\mathcal{P}_{K,k}' &= \left\{ P \in \mathcal{P}_{K,k} \mid P = \bigcap \{Q \in \mathcal{M}_{K,k} \mid P \le Q\} \right\}.
\end{align*}
With these definitions, the obvious analogue of Lemma~\ref{lem:intersection} holds for connected $G$. In the proof, one needs to work with the relative root system of $G$ with respect to a maximal $k$-split torus \cite[V.21]{Borel91}, \cite[\S 15, 16]{spr2}, but otherwise the argument goes through \emph{mutatis mutandis}. The conclusion of Lemma~\ref{lem:IntersectionParabolic} also holds when considering only $k$-defined cocharacters. Since $\mathcal{M}_{K,k} = \MK \cap \mathcal{P}_{K,k}$ and $\mathcal{P}_{K,k}' = \PK' \cap \mathcal{P}_{K,k}$, Lemma~\ref{lem:correct-sets} immediately implies its own rational analogue. Finally, \cite[Lem.~4.6]{BMRT:relative} shows that each time we use Lemma~\ref{lem:3.3} to conjugate a subgroup or cocharacter, the conjugating element can be taken to be a $k$-point.

The following are now the rational versions of Theorem~\ref{thm:minflags-general} and \ref{thm:condition-pk-prime}.

\begin{thm}
	\label{thm:minflags-rational}
	Let $K \le G$ be reductive $k$-defined algebraic groups with $G$ connected, and let $H$ be a subgroup of $G$. Then the following are equivalent.
	\begin{enumerate}
		\item $H$ is relatively $G$-completely reducible over $k$ with respect to $K$. \label{minflags-rational-i}
		\item Every member of $\mathcal{M}_{K,k}$ containing $H$ has an opposite in $\mathcal{M}_{K,k}$ containing $H$. \label{minflags-rational-ii}
		\item There is an R-Levi subgroup $L_{\mu}$ with $\mu \in Y_{k}(K)$, such that $H \le L_{\mu}$ and $H$ is relatively $L_{\mu}$-irreducible over $k$ with respect to $K \cap L_{\mu}$. \label{minflags-rational-iii}
	\end{enumerate}
\end{thm}

\begin{thm} \label{thm:condition-pk-prime-rational}
	Let $K \le G$ be reductive $k$-defined algebraic groups, and let $H$ be a subgroup of $G$. Then each member of $\mathcal{M}_{K,k}$ containing $H$ has an opposite in $\mathcal{M}_{K,k}$ containing $H$ if and only if each member of $\mathcal{P}_{K,k}'$ containing $H$ has an opposite in $\mathcal{P}_{K,k}'$ containing $H$.
\end{thm}

\begin{rem}
	By Remark~\ref{rem:not-all-k-def-parabs}, we could equally well take $\mathcal{P}_{K,k}$ to be the set of $k$-defined R-parabolic subgroups of the form $P_{\lambda}$ with $\lambda \in Y(K)$ and $\mathcal{M}_{K,k}$ to be the set of maximal members of this $\mathcal{P}_{K,k}$. However, working with R-parabolic and R-Levi subgroups corresponding to elements of $Y_{k}(K)$ allows our proofs to be more naturally generalized.
\end{rem}

\begin{proof}[Proof of Theorems~\ref{thm:condition-pk-prime-rational} and \ref{thm:minflags-rational}]
The rational analogue of Lemma~\ref{lem:correct-sets} gives one direction of Theorem~\ref{thm:condition-pk-prime-rational}. For the reverse direction, the rational analogues of Lemmas~\ref{lem:IntersectionParabolic} and \ref{lem:3.3} are almost sufficient for the proof of Theorem~\ref{thm:condition-pk-prime} to go through \emph{mutatis mutandis}. The only subtle point is when conjugating a cocharacter by an element of `$C_{P_{\nu}(K)}(H)$'; we need the conjugating element to lie in $C_{P_{\nu}(K)}(H)(k)$. That we can guarantee this follows from the conjugacy of maximal $k$-split tori in \cite[Lem.~2.12]{BHMR:cochars}.

Thus Theorem~\ref{thm:condition-pk-prime-rational} holds. In Theorem~\ref{thm:minflags-rational}, the equivalence of conditions \ref{minflags-rational-i} and \ref{minflags-rational-iii} follows from Proposition~\ref{prop:rational-317}. The implication \ref{minflags-rational-i} $\Rightarrow$ \ref{minflags-rational-ii} follows from the rational analogue of Lemma~\ref{lem:correct-sets}. As in the algebraically closed case, these implications do not use the fact that $G$ is connected. Finally, if $G$ is connected then the rational version of Lemma~\ref{lem:intersection} tells us that $\mathcal{P}_{K,k} = \mathcal{P}_{K,k}'$, and so the implication \ref{minflags-rational-ii} $\Rightarrow$ \ref{minflags-rational-i} follows from Theorem~\ref{thm:condition-pk-prime-rational}.
\end{proof}

\appendix

\section{Extended example: \texorpdfstring{$G_2$ in $\GL_7$}{G2 in GL7}} \label{sec:appendix}

We close by considering a non-trivial example of Corollary~\ref{cor:minflags-glv} (and hence Theorem~\ref{thm:minflags-general}) when $V$ is a faithful irreducible module for a simple algebraic group of exceptional type.

Fix the algebraically closed field $k$ and let $K$ be a simple algebraic group of type $G_2$ over $k$. Then $K$ can be realized as the group of invertible linear transformations of a $7$-dimensional $k$-vector space $V$ which preserve an alternating trilinear form and associated bilinear form (or quadratic form if $k$ has characteristic $2$) \cite[\S 3]{asch}. 

By \cite[Thms.~1,3]{asch}, the R-parabolic subgroups $P_{\lambda}(K)$ of $K$, with $\lambda \in Y(K)$, are precisely the stabilizers in $K$ of \emph{doubly singular} subspaces of $V$ of dimension $1$ or $2$, where \emph{doubly singular} means singular with respect to both the trilinear form and the bilinear (or quadratic) form. By \cite[Thm.~2]{asch}, $K$ is transitive on such subspaces of each dimension.

Now let $G = \GL(V)$. 
Elementary calculations with high weights show that a given maximal torus of $K$ is $G$-conjugate to the subtorus
\[
S =\{\diag(s,t,st^{-1}, 1, s^{-1} t, t^{-1}, s^{-1}) \mid s,t \in k^{*} \}
\]
of $G$, with respect to an appropriate basis of $V$. It then follows quickly that flags of subspaces of $V$ corresponding to a cocharacter $\lambda \in Y(S)$ involve intermediate subspaces of dimensions $\{2,5\}$, $\{1,3,4,6\}$ and $\{1,2,3,4,5,6\}$. In the first two cases, $P_{\lambda}(K)$ is a maximal parabolic subgroup of $K$ and $P_{\lambda} \in \MK$. In the third case $P_{\lambda}(K)$ and $P_{\lambda}$ are respectively Borel subgroups of $K$ and of $G$. The flags of type $\{2,5\}$ consist of doubly singular subspaces of $V$ and their annihilators, and $K$ is transitive on these. By \cite[\S\S 4.2]{asch} the flags of type $\{1,3,4,6\}$ can be described as those of the form $U \sse \Delta(U) \sse \Delta(U)^{\perp} \sse U^{\perp}$, where $U = \left<x\right>$ is $1$-dimensional and doubly singular and the $3$-dimensional subspace $\Delta(U)$ is defined as the radical of the bilinear form $(u,v) = f(x,u,v)$ with $f$ the $K$-stable trilinear form. Since $K$ is transitive on the subspaces $U$ it is also transitive on these flags.

Explicitly writing out condition \ref{minflags-general-ii} of Theorem~\ref{thm:minflags-general} yields the following criteria for relative complete reducibility in this scenario.
\begin{thm} \label{thm:g2}
With the above notation, a subgroup $H$ of $G = \GL(V)$ is relatively $G$-completely reducible with respect to the subgroup $K$ of type $G_{2}$ if and only if the following two conditions hold.
\begin{enumerate}
	\item \label{g2-i}If $H$ stabilizes a flag $U \sse U^{\perp}$ where $U$ is $2$-dimensional and doubly singular, then $H$ also stabilizes a flag $W \sse W^{\perp}$ where $W$ is $2$-dimensional and doubly singular, and $V = U \oplus W^{\perp} = W \oplus U^{\perp}$.
	\item \label{g2-ii}If $H$ stabilizes a flag $U \sse \Delta(U) \sse \Delta(U)^{\perp} \sse U^{\perp}$ where $U$ is $1$-dimensional and doubly singular, then $H$ stabilizes another such flag $W \sse \Delta(W) \sse \Delta(W)^{\perp} \sse W^{\perp}$ with
	\[ V = U \oplus W^{\perp} = W \oplus U^{\perp} = \Delta(U) \oplus \Delta(W)^{\perp} = \Delta(W) \oplus \Delta(U)^{\perp}. \]
\end{enumerate}
\end{thm}
For example, let $U \sse U^{\perp}$ be a flag as in \ref{g2-i} and let $H$ be any Levi subgroup of the corresponding parabolic subgroup of $\GL(V)$. Then $H \cong \GL_{2} \times \GL_{3} \times \GL_{2}$ and $H$ does not stabilize a flag of subspaces of the form in \ref{g2-ii}. Thus $H$ is relatively $G$-cr with respect to $G_{2}$ if and only if the complementary flag stabilized by $H$ also has the form described in \ref{g2-i}.

\begin{rem}
The construction of Chevalley groups of type $G_2$ given in \cite{asch} holds for arbitrary fields $k$, and produces the same description of the maximal parabolic subgroups thereof. Thus for split groups of type $G_2$ over any field $k$, the analogue of Theorem~\ref{thm:g2} holds, characterising relative complete reducibility over $k$. We thank the anonymous referee for pointing this out.
\end{rem}

%\section*{Acknowledgements}

\end{document}